\documentclass[12pt,reqno]{amsart}
\usepackage{a4wide}

\numberwithin{equation}{section}

\newtheorem{theorem}{Theorem}[section]
\newtheorem{proposition}[theorem]{Proposition}

\newtheorem{lemma}[theorem]{Lemma}

\theoremstyle{definition}

\theoremstyle{remark}

\newcommand{\ep}{\varepsilon}

\newcommand{\R}{\mathbb R_+^N}

\newcommand{\intr}{\int_{\mathbb R_+^N}}

\newcommand{\intpr}{\int_{\partial \mathbb R_+^N}}

\begin{document}
\title{Liouville Type Theorem for Some Nonlocal Elliptic Equations}
\author{Xiaohui  Yu$^{*}$}
\thanks{*School of Mathematics and Statistics, Shenzhen University, Shenzhen,
Guangdong 518060, PR China(xiaohui.yu@szu.edu.cn)} \maketitle

{\scriptsize  {\bf Abstract:} In this paper, we prove some Liouville theorem for the following elliptic equations involving nonlocal nonlinearity and nonlocal boundary value condition
$$
\left\{
  \begin{array}{ll}
  \displaystyle
-\Delta u(y)=\intpr \frac{
F(u(x',0))}{|(x',0)-y|^{N-\alpha}}dx'g(u(y)),    &y\in\R,    \\
\\ \displaystyle
 \frac{\partial u}{\partial \nu}(x',0)=\intr \frac{G(u(y))}{|(x',0)-y|^{N-\alpha}}\,dy f(u(x',0)), &(x',0)\in\partial \mathbb R_+^N,
\end{array}
\right.
$$
where $\mathbb R_+^N=\{x\in \mathbb R^N:x_N>0\}$, $f,g,F,G$ are some nonlinear functions. Under some assumptions on the nonlinear functions $f,g,F,G$, we will show that this equation doesn't possess nontrivial positive solution. We extend the Liouville theorems from local problems to nonlocal problem. We use the moving plane method to prove our result.

{\bf keywords:}\;\;Liouville type theorem, Moving plane method, Maximum principle, nonlocal problem.\\

\vspace{3mm} } \maketitle

\section {\bf Introduction}

In the studying of the existence of solutions for non-variational elliptic equations on bounded domain, we usually use the topological methods such as the Leray-
Schauder degree theory to get the existence results. In order to apply such a theory, a priori bound on the solution is usually needed. As far as we know, the blow-up method is the most powerful tool for proving a priori bound. The spirit of blow up method is straightforward. Suppose on the contrary that there exists a sequence of solutions $\{u_n\}$ with $M_n=u_n(x_n)=\|u_n\|_{L^\infty(\Omega)}\to \infty$, then we make a scaling on this sequence of solutions and get $v_n(x)=\frac 1{M_n}u_n(M_n^k x+x_n)$ which is bounded. Hence, by the regularity theory of elliptic equations, we can assume that $v_n\to v$ in $C_{loc}^{2,\gamma}(\Omega_\infty)$ with $\|v\|_{L^\infty}=1$ and satisfying some limit equation in $\Omega_\infty$, where either $\Omega_\infty=\mathbb R^N$ or $\Omega_\infty=\mathbb R_+^N$ depending on the speed of $x_n$ goes to the boundary of $\Omega$. On the other hand, if we can prove the limit equations don't possess nontrivial solutions, then we get a contradiction, hence the solutions of the original problem must be bounded. From the descriptions of the blow-up procedure, it is easy to see that a priori bound of elliptic equations on bounded domain is equivalent to the Liouville type theorems for the limit equations. Hence, from the past few decades, Liouville theorems for elliptic equations have attracted much attention of scientists and many results were obtained. The most remarkable
result on this aspect is the results in \cite{GS1}, in which the authors studied the nonexistence
results for the following elliptic equation
\begin{equation}\label{1.1}
-\Delta u=u^p\quad {\rm in}\quad \mathbb R^N,\quad u\geq 0.
\end{equation}
The authors proved, among other things, that problem \eqref{1.1} does not possess positive solution
provided $0<p<\frac{N+2}{N-2}$. Moreover, this result is optimal in the sense that for any $p\geq
\frac{N+2}{N-2}$, there are infinitely many positive solutions to problem
(\ref{1.1}). Thus the Sobolev exponent $ \frac {N+2}{N-2}$ is the
dividing exponent between existence and nonexistence of positive
solutions. For this reason, the exponent $ \frac {N+2}{N-2}$ is usually called the critical exponent for equation \eqref{1.1}. Later, in order to get a priori bound for some non-variational elliptic equations on bounded domains, the authors also studied a similar problem in half space
\begin{equation}\label{1.2}
\left\{
  \begin{array}{ll}
  \displaystyle
-\Delta u=u^p   &{\rm in}\quad \mathbb R_+^N,    \\
\displaystyle
\\ u=0 &{\rm on}\quad \partial \mathbb R_+^N
\end{array}
\right.
\end{equation}
in \cite{GS}. They proved that problem \eqref{1.2} also does not possess positive solution for $0<p<\frac {N+2}{N-2}$. Later, in order to simplify the proofs in \cite{GS} and \cite{GS1}, W.Chen and C.Li proved similar
results by using the moving plane method in \cite{CL}. The idea of \cite{CL} is very creative,
they proved the solution is symmetric in
every direction and with respect to every point, hence the solution
must be a constant, finally, they deduced from the equation that the constant must be zero. After the results of \cite{CL}, the moving plane method or its variant the moving sphere method were widely used in proving Liouville theorems for elliptic equations, see\cite{DG}\cite{Da}\cite{FF}\cite{GL}\cite{GL1}\cite{GLZ}\cite{LZ}\cite{Yu3}\cite{Yu4}\cite{Yu5} and the references therein. At the same time, the moving plane method was also widely used in proving Liouville theorems for integral equations, integral systems and fractional Laplacian equations, we refer the readers to \cite{CLL}\cite{CLO}\cite{CFL}\cite{CFY}\cite{CLZ}\cite{FC}\cite{Li}\cite{LZou}\cite{MC}\cite{Yu}\cite{Yu1} for more details.

Among the above works, we should mention the
 paper \cite{DG}. In this paper,
L.Damascelli and F.Gladiali studied the nonexistence of
weak positive solution for the following nonlinear problem with general
nonlinearity
\begin{equation}\label{1.3}
-\Delta u=f(u)\quad {\rm in}\quad \mathbb R^N,\quad u\geq 0,
\end{equation}
where $f$ is only assumed to be continuous rather than Lipschitz continuous. Since no Lipschitz continuous assumption is made on $f$, then the weak solution of \eqref{1.3} is usually not of $C^2$ class. Hence, the moving plane method based on the classical maximum principles in \cite{CL} does not work. In order to overcome this difficulty, the authors used some integral inequalities to substitute the classical maximum principles, an idea
originally due to S.Terracini's work \cite{Te1} and \cite{Te2}. Finally, using the moving plane method based on the integral inequalities, they
proved the only nonnegative solution for problem \eqref{1.3} is $u\equiv c$ providing $f$ is increasing and subcritical.
After the work of \cite{DG}, many nonexistence results for elliptic equations
with general nonlinearities were obtained, we refer the readers to
\cite{GL}\cite{GL1}\cite{GLZ}\cite{LZ}\cite{Yu}\cite{Yu1}\cite{Yu3}\cite{Yu4}.
Also, for Liouville theorems on nonlinear elliptic systems, we refer to
\cite{FF}\cite{Mi}\cite{SZ1}\cite{SZ2}\cite{SZ3}\cite{So} and etc..

In this paper, we study the nonexistence result for some elliptic equation involving nonlocal nonlinearity and nonlocal boundary value condition. The equation is
\begin{equation}\label{1.4}
\left\{
  \begin{array}{ll}
  \displaystyle
-\Delta u(y)=\intpr \frac{
F(u(x',0))}{|(x',0)-y|^{N-\alpha}}dx'g(u(y)),    &y\in\R,    \\
\\ \displaystyle
 \frac{\partial u}{\partial \nu}(x',0)=\intr \frac{G(u(y))}{|(x',0)-y|^{N-\alpha}}\,dy f(u(x',0)), &(x',0)\in\partial \mathbb R_+^N,
\end{array}
\right.
\end{equation}
where $N-2\leq \alpha<N$ and $f,g,F,G$ are some nonlinear functions.
We note that both the nonlinear term and the boundary value condition are nonlocal, so it is different from the equations in the references mentioned above. Moreover, up to now, it seems that there is few Liouville theorem on this kind of problem. To state our main results, we first need to give the definition of weak solution for problem \eqref{1.4}. Let $u\in W_{loc}^{1,2}(\overline{\mathbb R_+^N})\cap C^0(\overline{\mathbb R_+^N})$, if
\begin{eqnarray*}
&&\int_{\mathbb R_+^N}\nabla u(y)\nabla \varphi(y)\,dy \\
   &=&\int_{\mathbb R_+^N} \intpr \frac{
F(u(x',0))}{|(x',0)-y|^{N-\alpha}}dx'g(u(y))\varphi(y)\,dy\\&+&\int_{\partial \mathbb R_+^N}\intr \frac{G(u(y))}{|(x',0)-y|^{N-\alpha}}\,dy f(u(x',0))\varphi(x',0)\,dx'
\end{eqnarray*}
holds for any $\varphi(x)\in C_0^1(\overline{ \mathbb R_+^N})$, then we say $u$ is a weak solution for problem \eqref{1.4}.

With the above preparations and notations, we can state our main result now. We have the following Liouville type theorem.
\begin{theorem}\label{t 1.1}
Let $u\in W_{loc}^{1,2}(\overline{\mathbb R_+^N})\cap C^0(\overline{\mathbb R_+^N})$ be a nonnegative weak solution of
problem \eqref{1.4} with $N-2\leq \alpha <N$, where $f,g,F,G:[0,+\infty)\to [0,+\infty)$ are
continuous functions with the following properties

(i) $f(t),g(t),F(t), G(t)$ are nondecreasing in $(0,+\infty)$.

(ii) $h(t)=\frac{f(t)}{t^{\frac{\alpha}{N-2}}}$, $k(t)=\frac{g(t)}{t^{\frac{2+\alpha}{N-2}}}$, $H(t)=\frac{F(t)}{t^{\frac{N-2+\alpha}{N-2}}}$ and $K(t)=\frac{G(t)}{t^{\frac{N+\alpha}{N-2}}}$ and non-increasing in $(0,+\infty)$.

(iii) At least one of the functions $h,k,H,K$ is not a constant in $(0,\sup_{x\in \mathbb R_+^N}u(x))$ or $(0,\sup_{x\in \partial\mathbb R_+^N }u(x))$.

Then  $u\equiv c$ with $G(c)=F(c)=0$.
\end{theorem}

It is easy to find such nonlinear functions satisfying then requirements in the above theorem. For example, if we choose $f(t)=t^p$, $g(t)=t^q$, $F(t)=t^{p+1}$ and $G(t)=t^{q+1}$ with $0<p<\frac{\alpha}{N-2}$ and $0<q<\frac{2+\alpha}{N-2}$, then $f,g$ satisfies all the hypothesises in Theorem \ref{t 1.1}.

The rest of this paper is devoted to the proof of the above theorem. We still want to use the moving plane method to prove our result.
We prove the solution $u$ must be symmetric with respect to every point $(x'_p,0)\in \mathbb \partial \mathbb R_+^N$ and in every direction of $\mathbb R^{N-1}$. So the solution $u$ depends only on $x_N$. Finally, we deduce from the equation that $u$ must be a constant $c$ with $G(c)=F(c)=0$. In developing the moving plane method, since we don't know the decay behavior of $u$ at infinity, we need to make a Kelvin transformation on $u$, i.e., $v(x)=\frac 1{|x|^{N-2}}u(\frac x{|x|^2})$, then prove the symmetry property of $v$ instead. On the other hand, since we don't assume the solution is of $C^2$ class, we can't use the classical maximum principle in a differential form. Inspired by \cite{DG}\cite{Te1}\cite{Te2}, we use some integral inequality to substitute the maximum principle. In the
rest of this paper, we denote $C$ by a positive constant, which may
vary from line to line.

\section{\bf The integral inequality}

In this section, we make some preparations and prove some useful integral inequality. For convenience, we denote by
$$
m(y)=\int_{\partial \mathbb R_+^N}\frac{F(u(x',0))}{|(x',0)-y|^{N-\alpha}}\,dx'
$$
for $y$ in $\mathbb R_+^N$
and
$$
n(x')=\int_{ \mathbb R_+^N}\frac{G(u(y))}{|(x',0)-y|^{N-\alpha}}\,dy
$$
for $x'\in \mathbb R^{N-1}$, i.e., $(x',0)\in \partial \mathbb R_+^N$.
Then equation \eqref{1.4} can be written as
\begin{equation}\label{2.1}
\left\{
  \begin{array}{ll}
  \displaystyle
-\Delta u(y)=m(y)g(u(y)),    &y\in\R,    \\
\\ \displaystyle
 \frac{\partial u}{\partial \nu}(x',0)=n(x')f(u(x',0)), &(x',0)\in\partial \mathbb R_+^N.
\end{array}
\right.
\end{equation}
We want to use the moving plane method to prove our result. Since we don't know the decay of $u$, we need to make a Kelvin transformation on $u(y),m(y)$ and $n(x')$. For any $x_p=(x_p',0)\in \partial \mathbb R_+^N$, we define the Kelvin transformation $v,w,z$ of $u,m,n$ at $x_p$ as
\begin{equation}\label{2.2}
v(x)=\frac 1{|x-x_p|^{N-2}}u(\frac {x-x_p}{|x-x_p|^2}+x_p), \quad w(x)=\frac 1{|x-x_p|^{N-\alpha}}m(\frac {x-x_p}{|x-x_p|^2}+x_p)
\end{equation}
and
\begin{equation}\label{2.3}
z(x')=\frac 1{|x'-x_p'|^{N-\alpha}}n(\frac {x'-x_p'}{|x'-x_p'|^2}+x_p'),
\end{equation}
then it is easy to see that
\begin{equation}\label{2.4}
v(x)\leq \frac C{|x-x_p|^{N-2}},\ w(x)\leq \frac C{|x-x_p|^{N-\alpha}}\ {\rm and}\ z(x')\leq \frac C{|x'-x_p'|^{N-\alpha}}
\end{equation}
for $|x-x_p|\geq 1$ and they may have singularities at $x_p$. In the following, we can assume $x_p=0$ without loss of generality.
Moreover, a direct calculation shows that they satisfy the following equation
\begin{equation}\label{2.5}
\left\{
  \begin{array}{ll}
    \displaystyle
w(y)=\int_{\partial \mathbb R_+^N}\frac{H(|\bar x|^{N-2}v(\bar x,0))}{|(\bar x,0)-y|^{N-\alpha}}v(\bar x,0)^{\frac{N-2+\alpha}{N-2}}\,d\bar x,    &y\in\R\setminus\{0\},    \\
  \displaystyle
z(\bar x)=\int_{ \mathbb R_+^N}\frac{K(|y|^{N-2}v(y))}{|(\bar x,0)-y|^{N-\alpha}}v(y)^{\frac{N+\alpha}{N-2}}dy,   &\bar x\in \partial\R\setminus \{0\},    \\
  \displaystyle
-\Delta v(y)=w(y)k(|y|^{N-2}v(y))v(y)^{\frac{2+\alpha}{N-2}},    &y\in\R\setminus\{0\},    \\
\\ \displaystyle
 \frac{\partial v}{\partial \nu}(x',0)=z(x')h(|x'|^{N-2}v(x',0))v(x',0)^{\frac \alpha{N-2}}, &(x',0)\in\partial \mathbb R_+^N\setminus\{0\},
\end{array}
\right.
\end{equation}
where $H(t),K(t),h(t),k(t)$ are defined in assumption (ii) in Theorem \ref{t 1.1}.

In order to use the moving plane method, we need to introduce some notations which will be used later. For any $\lambda>0$, we denote by $\Sigma_\lambda=\{x\in \mathbb R_+^N|x_1>\lambda\}$, $\partial\Sigma_\lambda=\{x\in \partial\mathbb R_+^N|x_1>\lambda\}$, $T_\lambda=\{x\in \mathbb R_+^N|x_1=\lambda\}$. Moreover, for any $x\in \Sigma_\lambda$, we denote by $x^\lambda$ the reflection of $x$ with respect to $T_\lambda$, i.e., $x^\lambda=(2\lambda-x_1,x_2,...,x_N)$. In the following, we also define $v^\lambda(x)=v(x^\lambda)$ and $p^\lambda=(2\lambda,0,...,0)$.

With the above preparations, we have the following
\begin{lemma}\label{t 2.1}
For any $\lambda>0$, if we denote $\Sigma_\lambda^v=\{x\in \Sigma_\lambda| v(x)>v(x^\lambda)\}$ and $\partial\Sigma_\lambda^v=\{x\in \partial\Sigma_\lambda| v(x)>v(x^\lambda)\}$, then for any $y\in \Sigma_\lambda\setminus \{(2\lambda,0,...,0)\}$ and $\bar x\in \partial \Sigma_\lambda\setminus \{(2\lambda,0,...,0)\}$, we have
\begin{equation}\label{2.6}
w(y)-w(y^\lambda)\leq \int_{\partial \Sigma_\lambda^v}\frac 1{|(\bar x,0)-y|^{N-\alpha}}H(|\bar x|^{N-2}v(\bar x,0))[v(\bar x,0)^{\frac{N-2+\alpha}{N-2}}-v(\bar x^\lambda,0)^{\frac{N-2+\alpha}{N-2}}]\,d\bar x
\end{equation}
and
\begin{equation}\label{2.7}
z(\bar x)-z(\bar x^\lambda)\leq \int_{\Sigma_\lambda^v}\frac 1{|(x',0)-y|^{N-\alpha}}K(|y|^{N-2}v(y))[v(y)^{\frac{N+\alpha}{N-2}}-v(y^\lambda)^{\frac{N+\alpha}{N-2}}]\,dy.
\end{equation}
\end{lemma}
\begin{proof}
We only prove equation \eqref{2.7}, the proof of equation \eqref{2.6} is similar. By the equation of $z(\bar x)$, we have
$$
z(\bar x)=\int_{\Sigma_\lambda}\frac{K(|y|^{N-2}v(y))}{|(\bar x,0)-y|^{N-\alpha}}v(y)^{\frac{N+\alpha}{N-2}}dy+\int_{\Sigma_\lambda}\frac{K(|y^\lambda |^{N-2}v(y^\lambda))}{|(\bar x,0)-y^\lambda|^{N-\alpha}}v(y^\lambda)^{\frac{N+\alpha}{N-2}}dy,
$$
and
$$
z(\bar x^\lambda)=\int_{\Sigma_\lambda}\frac{K(|y|^{N-2}v(y))}{|(\bar x^\lambda,0)-y|^{N-\alpha}}v(y)^{\frac{N+\alpha}{N-2}}dy+\int_{\Sigma_\lambda}\frac{K(|y^\lambda |^{N-2}v(y^\lambda))}{|(\bar x^\lambda,0)-y^\lambda|^{N-\alpha}}v(y^\lambda)^{\frac{N+\alpha}{N-2}}dy.
$$
Since $|(\bar x,0)-y^\lambda |=|(\bar x^\lambda,0)-y|$ and $|(\bar x,0)-y|=|(\bar x^\lambda,0)-y^\lambda|$, we deduce from the above two equations that
\begin{equation}\label{2.8}
\begin{split}
&z(\bar x)-z(\bar x^\lambda)\\&=\int_{\Sigma_\lambda}(\frac 1{|(\bar x,0)-y|^{N-\alpha}}-\frac 1{|(\bar x,0)-y^\lambda|^{N-\alpha}})[K(|y|^{N-2}v(y))v(y)^{\frac{N+\alpha}{N-2}}-K(|y^\lambda|^{N-2}v(y^\lambda))v(y^\lambda)^{\frac{N+\alpha}{N-2}}]dy.
\end{split}
\end{equation}
For $y\in \Sigma_\lambda^v=\{x\in \Sigma_\lambda|v(x)>v(x^\lambda)\}$, we have $|y|^{N-2}v(y)>|y^\lambda|^{N-2}v(y^\lambda)$, hence from the monotonicity of $K(t)$, we deduce that
\begin{equation}\label{2.9}
K(|y|^{N-2}v(y))\leq K(|y^\lambda|^{N-2}v(y^\lambda)).
\end{equation}
On the other hand, for $y\in \Sigma_\lambda\setminus \Sigma_\lambda^v$, we have
\begin{equation}\label{2.10}
\begin{split}
K(|y|^{N-2}v(y))v(y)^{\frac{N+\alpha}{N-2}}&=\frac{G(|y|^{N-2}v(y))}{|y|^{N+\alpha}}\\&\leq \frac{G(|y|^{N-2}v(y^\lambda))}{|y|^{N+\alpha}}\\&=
\frac{G(|y|^{N-2}v(y^\lambda))}{[|y|^{N-2}v(y^\lambda)]^{\frac{N+\alpha}{N-2}}}v(y^\lambda)^{\frac{N+\alpha}{N-2}}\\&\leq
\frac{G(|y^\lambda|^{N-2}v(y^\lambda))}{[|y^\lambda|^{N-2}v(y^\lambda)]^{\frac{N+\alpha}{N-2}}}v(y^\lambda)^{\frac{N+\alpha}{N-2}}\\&=
K(|y^\lambda|^{N-2}v(y^\lambda))v(y^\lambda)^{\frac{N+\alpha}{N-2}}.
\end{split}
\end{equation}
Hence we deduce from equations \eqref{2.8}\eqref{2.9}\eqref{2.10} that
\begin{equation}\label{2.11}
\begin{split}
&z(\bar x)-z(\bar x^\lambda)\\&\leq \int_{\Sigma_\lambda^v}\frac 1{|(\bar x,0)-y|^{N-\alpha}}K(|y|^{N-2}v(y))[v(y)^{\frac{N+\alpha}{N-2}}-v(y^\lambda)^{\frac{N+\alpha}{N-2}}]dy.
\end{split}
\end{equation}
This completes the proof of equation \eqref{2.7}.
\end{proof}
\begin{lemma}\label{t 2.2}
Under the assumptions of Theorem \ref{t 1.1} and $v,w,z$ defined as above, then for any fixed $\lambda>0$, the functions $v$ and
$(v-v^\lambda)^+$ belong to $ L^{2^*}(\Sigma_\lambda)\cap
L^\infty(\Sigma_\lambda)$ with $2^*=\frac{2N}{N-2}$. Furthermore,
if we denote $\Sigma_\lambda^v$ and $\partial \Sigma_\lambda^v$ as above, then there exists
$C_\lambda>0$, which is nonincreasing in $\lambda$, such that
\begin{equation}\label{2.12}
\begin{split}
&\int_{\Sigma_\lambda}|\nabla (v-v^\lambda)^+|^2\,dx\\&\leq C_\lambda [\|w(y)\|_{L^{\frac{2N}{N-\alpha}}(\Sigma_\lambda^v)}\|v(y)\|^{\frac{4+\alpha-N}{N-2}}_{L^{\frac{2N}{N-2}}(\Sigma_\lambda^v)}+\|v(\bar x,0)\|_{L^{\frac{2(N-1)}{N-2}}(\partial \Sigma_\lambda^v)}^{\frac \alpha{N-2}}\|v(y)\|_{L^{\frac{2N}{N-2}}(\Sigma_\lambda^v)}^{\frac{2+\alpha}{N-2}}
\\&+\|z(x')\|_{L^{\frac{2(N-1)}{N-\alpha}}(\partial\Sigma_\lambda^v)}\|v(x',0)\|_{L^{\frac{2(N-1)}{N-2}}(\partial\Sigma_\lambda^v)}^{\frac{2+\alpha-N}{N-2}}]\cdot\int_{\Sigma_\lambda}|\nabla (v-v^\lambda)^+|^2\,dx.
\end{split}
\end{equation}
\end{lemma}
\begin{proof}
First, we note that since $\lambda>0$, then there exists $r>0$ such that
$\Sigma_\lambda\subset  \mathbb R_+^N\setminus B_r^+(0)$, where $B_r^+(0)=\{x\in \mathbb R_+^N||x|<r,\ x_N>0\}$. Hence, we deduce from the definition of $v(x)$ and the decay behavior of $v(x)$ at infinity that
$$
v,(v-v^\lambda)^+\in L^{2^*}(\Sigma_\lambda)\cap
L^\infty(\Sigma_\lambda).
$$

Second, in order to deal with the possible singularity of $v(x),w(x)$ and $z(x')$ at zero, we need to introduce some cut-off function. Let $\eta=\eta_\varepsilon\in C( \mathbb R_+^N, [0,1])$ be a radial function such that
 $$
\eta(x)= \left\{
\begin{array}{ll}
1, & \quad  2\varepsilon\leq |x-p^\lambda|\leq \frac 1\varepsilon,\\
 \\0, & \quad \ |x-p^\lambda|<\varepsilon,\ |x-p^\lambda|> \frac 2\varepsilon.
\end{array}
\right.
$$
Moreover, we require that $|\nabla \eta|\leq \frac 2\varepsilon$ for $\varepsilon<|x-p^\lambda|<2\varepsilon$ and $|\nabla \eta|\leq  2\varepsilon$ for $\frac 1\varepsilon<|x-p^\lambda|<\frac
2\varepsilon$. It is easy to see that such cut-off function exists. For convenience, we also define $\varphi=\varphi_\varepsilon=\eta^2_\varepsilon(v-v^\lambda)^+$ and
$\psi=\psi_\varepsilon=\eta_\varepsilon(v-v^\lambda)^+$, then it is easy to see that $\varphi$ and $\psi$ is well-defined on $\Sigma_\lambda$. Moreover, a simple calculation shows that
$$
|\nabla \psi|^2=\nabla(v-v^\lambda)^+\nabla
\varphi+[(v-v^\lambda)^+]^2|\nabla \eta|^2.
$$
With the above preparations, we have
\begin{equation}\label{2.13}
\begin{split}
&\int_{\Sigma_\lambda\cap \{2\varepsilon\leq |x-p^\lambda|\leq \frac 1\varepsilon\}}|\nabla(v-v^\lambda)^+|^2\,dx\\&\leq \int_{\Sigma_\lambda}|\nabla \psi(x)|^2\,dx
\\&\leq \int_{\Sigma_\lambda}\nabla(v-v^\lambda)\nabla \varphi\,dx+\int_{\Sigma_\lambda}[(v-v^\lambda)^+]^2|\nabla\eta_\varepsilon|^2\,dx\\&=
\int_{\Sigma_\lambda}-\Delta(v-v^\lambda)\varphi(y)\,dy+\int_{\partial \Sigma_\lambda}\frac{\partial (v-v^\lambda)}{\partial \nu}\varphi(x',0)\,dx'+I_\varepsilon
\\&= \int_{\Sigma_\lambda}[w(y)k(|y|^{N-2}v(y))v(y)^{\frac{2+\alpha}{N-2}}-
w(y^\lambda)k(|y^\lambda|^{N-2}v(y^\lambda))v(y^\lambda)^{\frac{2+\alpha}{N-2}}][v(y)-v(y^\lambda)]^+\eta_{\varepsilon}^2\,dy\\&+\int_{\partial \Sigma_\lambda} [z(x')h(|x'|^{N-2}v(x',0))v(x',0)^{\frac{\alpha}{N-2}}-z(x'^\lambda)h(|x'^\lambda|^{N-2}v(x'^\lambda,0))v(x'^\lambda,0)^{\frac {\alpha}{N-2}}]\\&\cdot [v(x',0)-v(x'^\lambda,0)]^+\eta_{\varepsilon}^2\,dx'
+I_\varepsilon\\&=\int_{\Sigma_\lambda^v}[w(y)k(|y|^{N-2}v(y))v(y)^{\frac{2+\alpha}{N-2}}-
w(y^\lambda)k(|y^\lambda|^{N-2}v(y^\lambda))v(y^\lambda)^{\frac{2+\alpha}{N-2}}][v(y)-v(y^\lambda)]^+\eta_{\varepsilon}^2\,dy\\&+\int_{\partial \Sigma_\lambda^v} [z(x')h(|x'|^{N-2}v(x',0))v(x',0)^{\frac \alpha{N-2}}-z(x'^\lambda)h(|x'^\lambda|^{N-2}v(x'^\lambda,0))v(x'^\lambda,0)^{\frac \alpha{N-2}}]\\&\cdot[v(x',0)-v(x'^\lambda,0)]^+\eta_{\varepsilon}^2\,dx'
+I_\varepsilon\\&=I+II+I_\varepsilon,
\end{split}
\end{equation}
where $I_\varepsilon=\int_{\Sigma_\lambda}[(v(y)-v(y^\lambda)^+)]^2|\nabla\eta_\varepsilon|^2\,dy$.

To estimate the integral $I$, we first note that
$$
k(|y|^{N-2}v(y))\leq k(|y^\lambda|^{N-2}v(y^\lambda))
$$
for $y\in \Sigma_\lambda^v$ by the monotonicity assumption of $k$. Now we divide the integral domain into two parts, the first part is $D_1=\{y|v(y)>v(y^\lambda),w(y)>w(y^\lambda)\}$ and the second is $D_2=\{y|v(y)>v(y^\lambda),w(y)\leq w(y^\lambda)\}$.

For $y\in D_1$, we have
\begin{equation}\label{2.14}
\begin{split}
&w(y)k(|y|^{N-2}v(y))v(y)^{\frac{2+\alpha}{N-2}}-
w(y^\lambda)k(|y^\lambda|^{N-2}v(y^\lambda))v(y^\lambda)^{\frac{2+\alpha}{N-2}}\\&=[w(y)-w(y^\lambda)]k(|y|^{N-2}v(y))v(y)^{\frac{2+\alpha}{N-2}}
\\&+w(y^\lambda)[k(|y|^{N-2}v(y))v(y)^{\frac{2+\alpha}{N-2}}-k(|y^\lambda|^{N-2}v(y^\lambda))v(y^\lambda)^{\frac{2+\alpha}{N-2}}]\\&\leq [w(y)-w(y^\lambda)]k(|y|^{N-2}v(y))v(y)^{\frac{2+\alpha}{N-2}}
\\&+w(y)k(|y|^{N-2}v(y))[v(y)^{\frac{2+\alpha}{N-2}}-v(y^\lambda)^{\frac{2+\alpha}{N-2}}].
\end{split}
\end{equation}
While for $y\in D_2$, we have
\begin{equation}\label{2.15}
\begin{split}
&w(y)k(|y|^{N-2}v(y))v(y)^{\frac{2+\alpha}{N-2}}-
w(y^\lambda)k(|y^\lambda|^{N-2}v(y^\lambda))v(y^\lambda)^{\frac{2+\alpha}{N-2}}\\&\leq w(y)[k(|y|^{N-2}v(y))v(y)^{\frac{2+\alpha}{N-2}}-k(|y^\lambda|^{N-2}v(y^\lambda))v(y^\lambda)^{\frac{2+\alpha}{N-2}}]\\&
\leq
w(y)k(|y|^{N-2}v(y))[v(y)^{\frac {2+\alpha}{N-2}}-v(y^\lambda)^{\frac {2+\alpha}{N-2}}].
\end{split}
\end{equation}

Similarly, in order to estimate the integral $II$, we note that
$$
h(|x'|^{N-2}v(x',0))\leq h(|x'^\lambda|^{N-2}v(x'^\lambda,0))
$$
for $x'\in \partial\Sigma_\lambda^v$ by the monotonicity assumption of $h$.
As before, we still divide the integral domain $\partial \Sigma_\lambda^v$ into $D_3\cup D_4$, where $D_3=\{x'\in \partial \Sigma_\lambda|v(x',0)>v(x'^\lambda,0),z(x')>z(x'^\lambda)\}$ and $D_4=\{x'\in \partial \Sigma_\lambda|v(x',0)>v(x'^\lambda,0),z(x')\leq z(x'^\lambda)\}$. We can infer as the above procedure that for $x'\in D_3$
\begin{equation}\label{2.16}
\begin{split}
&z(x')h(|x'|^{N-2}v(x',0))v(x',0)^{\frac{\alpha}{N-2}}-
z(x'^\lambda)h(|x'^\lambda|^{N-2}v(x'^\lambda,0))v(x'^\lambda,0)^{\frac{\alpha}{N-2}}\\&\leq h(|x'|^{N-2}v(x',0))\{z(x')[v(x',0)^{\frac{\alpha}{N-2}}-v(x'^\lambda,0)^{\frac \alpha{N-2}}]+v(x',0)^{\frac \alpha{N-2}}[z(x')-z(x'^\lambda)]\}.
\end{split}
\end{equation}
While for $x'\in D_4$, the following inequality holds
\begin{equation}\label{2.17}
\begin{split}
&z(x')h(|x'|^{N-2}v(x',0))v(x',0)^{\frac{\alpha}{N-2}}-
z(x'^\lambda)h(|x'^\lambda|^{N-2}v(x'^\lambda,0))v(x'^\lambda,0)^{\frac{\alpha}{N-2}}\\&\leq
z(x')h(|x'|^{N-2}v(x',0))[v(x',0)^{\frac{\alpha}{N-2}}-v(x'^\lambda,0)^{\frac{\alpha}{N-2}}].
\end{split}
\end{equation}

Hence, we deduce from equations \eqref{2.13} \eqref{2.14} \eqref{2.15} \eqref{2.16} and \eqref{2.17} that
\begin{equation}\label{2.18}
\begin{split}
&\int_{\Sigma_\lambda\cap \{2\varepsilon\leq |x-p^\lambda|\leq \frac 1\varepsilon\}}|\nabla(v-v^\lambda)^+|^2\,dx\\&\leq
I_\varepsilon+C\int_{\Sigma_\lambda^v}w(y)k(|y|^{N-2}v(y))v(y)^{\frac{4+\alpha-N}{N-2}}[(v(y)-v(y^\lambda))^+]^2\eta_\ep(y)^2\,dy\\&+\int_{\Sigma_\lambda^v}
k(|y|^{N-2}v(y))v(y)^{\frac{2+\alpha}{N-2}}[w(y)-w(y^\lambda)]^+[v(y)-v(y^\lambda)]^+\eta_\ep(y)^2\,dy\\&+
C\int_{\partial \Sigma_\lambda^v}h(|x'|^{N-2}v(x',0))z(x')v(x',0)^{\frac{2+\alpha-N}{N-2}}[(v(x',0)-v(x'^\lambda,0))^+]^2\eta_\ep(x',0)^2\,dx'\\&
+\int_{\partial\Sigma_\lambda^v}h(|x'|^{N-2}v(x',0))v(x',0)^{\frac \alpha {N-2}}[z(x')-z(x'^\lambda)]^+[v(x',0)-v(x'^\lambda,0)]^+\eta_\ep(x',0)^2\,dx'\\&=
I_\ep+A+B+C+D.
\end{split}
\end{equation}

We claim that $I_\ep\to 0$ as $\ep\to 0$. In fact, if we denote
$D_\varepsilon^1=\{x\in \Sigma_\lambda|\varepsilon<|x-p^\lambda|<2\varepsilon\}$ and
$D_\varepsilon^2=\{x\in \Sigma_\lambda|\frac 1\varepsilon<|x-p^\lambda|<\frac
2\varepsilon\}$, then we have
$$
\int_{D_\varepsilon^1}|\nabla \eta|^N\,dx\leq C\frac 1{\varepsilon^N}\cdot \varepsilon^N=C.
$$
Similarly, we have
$$
\int_{D_\varepsilon^2}|\nabla \eta|^N\,dx\leq C\varepsilon^N\cdot \frac 1{\varepsilon^N}=C.
$$
Hence, we deduce from H\"older inequality and the fact that $[v(y)-v(y^\lambda)]^+\in L^{2^*}(\Sigma_\lambda)$ that
$$
I_\varepsilon\leq (\int_{D_\varepsilon^1\cup
D_\varepsilon^2}[(v-v^\lambda)^+]^{2^*}\,dx)^{\frac 2{2^*}}\cdot
(\int_{\Sigma_\lambda}|\nabla \eta|^N\,dx)^{\frac 2N} \to 0
$$
as $\varepsilon \to 0$.

Next, we estimate the integrals $A,B,C,D$. As for $A$, we infer from the monotonicity of $k$ and the decay of $v$ that there exists $C_\lambda>0$, which is non-increasing in $\lambda$, such that
\begin{equation}\label{2.19}
\begin{split}
A&\leq C_\lambda\int_{\Sigma_\lambda^v}w(y)v(y)^{\frac{4+\alpha-N}{N-2}}[(v(y)-v(y^\lambda))^+]^2\,dy\\&\leq C_\lambda (\int_{\Sigma_\lambda^v}w(y)^{\frac{2N}{N-\alpha}}\,dy)^{\frac{N-\alpha}{2N}}\cdot (\int_{\Sigma_\lambda^v}v(y)^{\frac{2N}{N-2}}\,dy)^{\frac{4+\alpha-N}{2N}}(\int_{\Sigma_\lambda^v}[(v(y)-v(y^\lambda))^+]^{\frac{2N}{N-2}})^{\frac{N-2}{N}}.
\end{split}
\end{equation}

Similarly, for  the integral $B$, we infer from Lemma \ref{t 2.1}, the Hardy-Littlewood-Sobolev inequality in \cite{DZ} and H\"older inequality that
\begin{equation}\label{2.20}
\begin{split}
B&\leq C_\lambda\int_{\Sigma_\lambda^v}v(y)^{\frac{2+\alpha}{N-2}}[w(y)-w(y^\lambda)]^+[v(y)-v(y^\lambda)]^+\,dy\\&\leq
C_\lambda\int_{\Sigma_\lambda^v}\int_{\partial\Sigma_\lambda^v}\frac{1}{|(\bar x,0)-y|^{N-\alpha}}v(\bar x,0)^{\frac \alpha{N-2}}[v(\bar x,0)-v(\bar x^\lambda,0)]^+\,d\bar xv(y)^{\frac{2+\alpha}{N-2}}[v(y)-v(y^\lambda)]^+\,dy\\&\leq C_\lambda\|v(\bar x,0)^{\frac{\alpha}{N-2}}(v(\bar x,0)-v(\bar x^\lambda,0))^+\|_{L^{\frac{2(N-1)}{N-2+\alpha}}(\partial \Sigma_\lambda^v)}\|v(y)^{\frac{2+\alpha}{N-2}}[v(y)-v(y^\lambda)]^+\|_{L^{\frac{2N}{N+\alpha}}(\Sigma_\lambda^v)}\\&\leq C_\lambda
\|v(\bar x,0)\|_{L^{\frac{2(N-1)}{N-2}}}^{\frac \alpha{N-2}}\|(v(\bar x,0)-v(\bar x^\lambda,0))^+\|_{L^{\frac{2(N-1)}{N-2}}(\partial \Sigma_\lambda^v)}\|v(y)\|_{L^{\frac{2N}{N-2}}(\Sigma_\lambda^v)}^{\frac{2+\alpha}{N-2}}\|(v(y)-v(y^\lambda))^+\|_{L^{\frac{2N}{N-2}}(\Sigma_\lambda^v)}.
\end{split}
\end{equation}

For the integral $C$, we infer from H\"older inequality that
\begin{equation}\label{2.21}
\begin{split}
C&\leq C_\lambda\int_{\partial \Sigma_\lambda^v}z(x')v(x',0)^{\frac{2+\alpha-N}{N-2}}[(v(x',0)-v(x'^\lambda,0))^+]^2\,dx'\\&\leq
C_\lambda (\int_{\partial \Sigma_\lambda^v} z(x')^{\frac{2(N-1)}{N-\alpha}}\,dx')^{\frac{N-\alpha}{2(N-1)}}(\int_{\partial \Sigma_\lambda^v}v(x',0)^{\frac{2(N-1)}{N-2}}\,dx')^{\frac{2+\alpha-N}{2(N-1)}}\\&\cdot(\int_{\partial \Sigma_\lambda^v} [(v(x',0)-v(x'^\lambda,0))^+]^{\frac{2(N-1)}{N-2}}\,dx')^{\frac{N-2}{N-1}}.
\end{split}
\end{equation}

Finally, for the integral $D$, we infer from Lemma \ref{t 2.1}, the Hardy-Littlewood-Sobolev inequality in half space in \cite{DZ} and H\"older inequality that
\begin{equation}\label{2.22}
\begin{split}
D&\leq C_\lambda\int_{\partial \Sigma_\lambda^v}v(x',0)^{\frac \alpha{N-2}}\int_{\Sigma_\lambda^v}\frac 1{|(x',0)-y|^{N-\alpha}}v(y)^{\frac{2+\alpha}{N-2}}[v(y)-v(y^\lambda)]^+\,dy[v(x',0)-v(x'^\lambda,0)]^+\,dx'\\&\leq C_\lambda \|v(x',0)^{\frac{\alpha}{N-2}}[v(x',0)-v(x'^\lambda,0)]^+\|_{L^{\frac{2(N-1)}{N+\alpha-2}}(\partial \Sigma_\lambda^v)}\|v(y)^{\frac{2+\alpha}{N-2}}[v(y)-v(y^\lambda)]^+\|_{L^{\frac{2N}{N+\alpha}}(\Sigma_\lambda^v)}\\&\leq C_\lambda\|v(x',0)\|_{L^{\frac{2(N-1)}{N-2}}(\partial \Sigma_\lambda^v)}^{\frac \alpha{N-2}}\|(v(x',0)-v(x'^\lambda,0))^+\|_{L^{\frac{2(N-1)}{N-2}}(\partial \Sigma_\lambda^v)}\\&\cdot\|v(y)\|_{L^{\frac{2N}{N-2}}(\Sigma_\lambda^v)}^{\frac{2+\alpha}{N-2}}\|(v(y)-v(y^\lambda))^+\|_{L^{\frac{2N}{N-2}}(\Sigma_\lambda^v)}.
\end{split}
\end{equation}

At last, let $\ep\to 0$ in equation \eqref{2.18} and use equations \eqref{2.19}\eqref{2.20}\eqref{2.21}\eqref{2.22}, the dominated convergence theorem, Sobolev inequality and Sobolev trace inequality, then we get
\begin{equation}\label{2.23}
\begin{split}
&\int_{\Sigma_\lambda}|\nabla (v-v^\lambda)^+|^2\,dx\\&\leq C_\lambda [\|w(y)\|_{L^{\frac{2N}{N-\alpha}}(\Sigma_\lambda^v)}\|v(y)\|^{\frac{4+\alpha-N}{N-2}}_{L^{\frac{2N}{N-2}}(\Sigma_\lambda^v)}+\|v(\bar x,0)\|_{L^{\frac{2(N-1)}{N-2}}(\partial \Sigma_\lambda^v)}^{\frac \alpha{N-2}}\|v(y)\|_{L^{\frac{2N}{N-2}}(\Sigma_\lambda^v)}^{\frac{2+\alpha}{N-2}}
\\&+\|z(x')\|_{L^{\frac{2(N-1)}{N-\alpha}}(\partial\Sigma_\lambda^v)}\|v(x',0)\|_{L^{\frac{2(N-1)}{N-2}}(\partial\Sigma_\lambda^v)}^{\frac{2+\alpha-N}{N-2}}
]\cdot\int_{\Sigma_\lambda}|\nabla (v-v^\lambda)^+|^2\,dx.
\end{split}
\end{equation}
This completes the proof of this lemma.

\end{proof}

To end this section, we want to make some comments on this lemma. As we have notified in the introduction, because the nonlinear term and the boundary value condition of the equation are not assumed to be Lipschitz continuous, the solution is not of $C^2$ class in general, so we can't use the usual moving plane method based on the maximum principle in a differential form. Thanks to Lemma \ref{t 2.2}, it can play the same role as the maximum principle does. In fact, if we can prove that
\begin{equation}\label{2.24}
\begin{split}
&C_\lambda [\|w(y)\|_{L^{\frac{2N}{N-\alpha}}(\Sigma_\lambda^v)}\|v(y)\|^{\frac{4+\alpha-N}{N-2}}_{L^{\frac{2N}{N-2}}(\Sigma_\lambda^v)}+\|v(\bar x,0)\|_{L^{\frac{2(N-1)}{N-2}}(\partial \Sigma_\lambda^v)}^{\frac \alpha{N-2}}\|v(y)\|_{L^{\frac{2N}{N-2}}(\Sigma_\lambda^v)}^{\frac{2+\alpha}{N-2}}
\\&+\|z(x')\|_{L^{\frac{2(N-1)}{N-\alpha}}(\partial\Sigma_\lambda^v)}\|v(x',0)\|_{L^{\frac{2(N-1)}{N-2}}(\partial\Sigma_\lambda^v)}^{\frac{2+\alpha-N}{N-2}}
]<1,
\end{split}
\end{equation}
then we infer from equation \eqref{2.12} that $\int_{\Sigma_\lambda}|\nabla (v-v^\lambda)^+|^2\,dx=0$, or $v(y)\leq v(y^\lambda)$ in $\Sigma_\lambda$, the same conclusion as the maximum principles imply.

\section{\bf Proof of Theorem \ref{t 1.1}}

With the above preparations, we can prove Theorem \ref{1.1} now. We use the moving plane method based on integral inequality to prove our result. First, we will show that this procedure can be started from some place. More precisely, we have the following
\begin{lemma}\label{3.1}
Under the assumptions of Theorem \ref{t 1.1}, there exists
$\lambda_0>0$, such that for all $\lambda\geq \lambda_0$, we have
$w(y)\leq w(y^\lambda), v(y)\leq v(y^\lambda)$ and $z(x')\leq z(x'^\lambda)$ in $\Sigma_\lambda$ and $\partial \Sigma_\lambda$ respectively.
\end{lemma}
\begin{proof}
The results of this lemma is a direct consequence of Lemma \ref{t 2.1} and Lemma \ref{t 2.2}. In fact, by the decay of $w(y), v(y)$ and $z(x')$, see equation \eqref{2.4}, we can find $\lambda_0>0$ large enough, such that for all $\lambda>\lambda_0$, we have
\begin{equation}\label{}
\begin{split}
&C_\lambda [\|w(y)\|_{L^{\frac{2N}{N-\alpha}}(\Sigma_\lambda^v)}\|v(y)\|^{\frac{4+\alpha-N}{N-2}}_{L^{\frac{2N}{N-2}}(\Sigma_\lambda^v)}+\|v(\bar x,0)\|_{L^{\frac{2(N-1)}{N-2}}(\partial \Sigma_\lambda^v)}^{\frac \alpha{N-2}}\|v(y)\|_{L^{\frac{2N}{N-2}}(\Sigma_\lambda^v)}^{\frac{2+\alpha}{N-2}}
\\&+\|z(x')\|_{L^{\frac{2(N-1)}{N-\alpha}}(\partial\Sigma_\lambda^v)}\|v(x',0)\|_{L^{\frac{2(N-1)}{N-2}}(\partial\Sigma_\lambda^v)}^{\frac{2+\alpha-N}{N-2}}
]<\frac 12,
\end{split}
\end{equation}
so we can conclude from Lemma \ref{t 2.2} that $v(y)\leq v(y^\lambda)$ in $\Sigma_\lambda$. Substitute the result into Lemma \ref{t 2.1}, then we get $w(y)\leq w(y^\lambda)$ and $z(x')\leq z(x'^\lambda)$ in $\Sigma_\lambda$ and $\partial \Sigma_\lambda$ respectively. The result of this lemma follows.
\end{proof}

Based on the above result, we can move the plane $T_{\lambda_0}$ from the right to the left as long as $w(y)\leq w(y^\lambda), v(y)\leq v(y^\lambda)$ and $z(x')\leq z(x'^\lambda)$ in $\Sigma_\lambda$ and $\partial \Sigma_\lambda$ respectively. We suppose the procedure stops at some $\lambda_1$. More precisely, we define
$$
\lambda_1= \inf\{\lambda|w(y)\leq w(y^{\bar\lambda}),\ v(y)\leq v(y^{\bar\lambda})\ {\rm in}\ \Sigma_{\bar\lambda},\ z(x')\leq z(x'^{\bar\lambda})\ {\rm in}\ \partial\Sigma_{\bar\lambda}\ {\rm for\ all\ }\bar\lambda\geq \lambda\},
$$
then we have the following
\begin{lemma}\label{t 3.2}
If $\lambda_1>0$, then we have $w(y)\equiv w(y^{\lambda_1}),\ v(y)\equiv v(y^{\lambda_1})$ in $\Sigma_{\lambda_1}$ and $z(x')\equiv z(x'^{\lambda_1})$ in $\partial \Sigma_{\lambda_1}$.
\end{lemma}
\begin{proof}
We prove the conclusion of this this lemma by contradiction. Suppose on the contrary that $w(y)\not\equiv w(y^{\lambda_1})$ or $v(y)\not\equiv v(y^{\lambda_1})$ or $z(x')\not\equiv z(x'^{\lambda_1})$, then we claim that the plane $T_{\lambda_1}$ can be moved to the left a little, that is, there exists $\delta_0>0$, such that $w(y)\leq w(y^\lambda),\ v(y)\leq v(y^\lambda)$ in $\Sigma_\lambda$ and $z(x')\leq z(x'^\lambda)$ in $\partial \Sigma_\lambda$ for all $\lambda\in [\lambda_1-\delta_0,\lambda_1]$. If the claim is true, then it contradicts the definition of $\lambda_1$, hence the result of this lemma holds. So in the following, we only need to prove the claim.

First, we have
$w(y)\leq w(y^{\lambda_1} ),\ v(y)\leq v(y^{\lambda_1})$ and $z(x')\leq z(x'^{\lambda_1})$ by the continuity of $w,v,z$ on $\lambda$. Moreover, by the monotonicity assumptions on $g$ and $k$, we have
\begin{eqnarray*}
&&w(y)k(|y|^{N-2}v(y))v(y)^{\frac{2+\alpha}{N-2}}\\&=& w(y) \frac{g(|y|^{N-2}v(y))}{|y|^{2+\alpha}}\\&\leq & w(y^{\lambda_1}) \frac{g(|y|^{N-2}v(y^{\lambda_1}))}{|y|^{2+\alpha}}\\&=& w(y^{\lambda_1}) \frac{g(|y|^{N-2}v(y^{\lambda_1}))}{[|y|^{N-2}v(y^{\lambda_1})]^{\frac{2+\alpha}{N-2}}}v(y^{\lambda_1})^{\frac{2+\alpha}{N-2}}\\&\leq &w(y^{\lambda_1}) \frac{g(|y^{\lambda_1}|^{N-2}v(y^{\lambda_1}))}{[|y^{\lambda_1}|^{N-2}v(y^{\lambda_1})]^{\frac{2+\alpha}{N-2}}}v(y^{\lambda_1})^{\frac{2+\alpha}{N-2}}
\\&=&
w(y^{\lambda_1})k(|y^{\lambda_1}|^{N-2}v(y^{\lambda_1}))v(y^{\lambda_1})^{\frac{2+\alpha}{N-2}},
\end{eqnarray*}
which further implies that
$$
-\Delta v(y)\leq -\Delta v(y^{\lambda_1})
$$
in $\Sigma_{\lambda_1}$.
Since $w(y)\not\equiv w(y^{\lambda_1})$ or $v(y)\not\equiv v(y^{\lambda_1})$ or $z(x')\not\equiv z(x'^{\lambda_1})$, then the strong maximum principle implies that
$v(y)< v(y^{\lambda_1})$ in $\Sigma_{\lambda_1}$. Further more, it implies $w(y)< w(y^{\lambda_1} )$ in $\Sigma_{\lambda_1}$ and $z(x')< z(x'^{\lambda_1})$ in $\partial\Sigma_{\lambda_1}$.

On the other hand, since $\frac
1{|x|^{2N}}\chi_{ \Sigma_\lambda^v}\to 0$, $\frac
1{|x'|^{2(N-1)}}\chi_{\partial\Sigma_\lambda^v}\to 0$ a.e. as $\lambda\to\lambda_1$ and there exists $\delta_1>0$ such that
$\frac
1{|x|^{2N}}\chi_{\Sigma_\lambda^v}\leq \frac
1{|x|^{2N}}\chi_{\Sigma_{\lambda_1-\delta_1}^w}$, $\frac
1{|x＆|^{2(N-1)}}\chi_{\partial\Sigma_\lambda^v}\leq \frac
1{|x'|^{2(N-1)}}\chi_{\partial\Sigma_{\lambda_1-\delta_1}^w}$ for $\lambda\in
[\lambda_1-\delta_1, \lambda_1]$, then it follows from the dominated convergence theorem that
$$
\int_{\Sigma_{\lambda}^v}\frac 1{|x|^{2N}}\,dx\to 0
$$
and
$$
\int_{\partial \Sigma_{\lambda}^v}\frac 1{|x'|^{2(N-1)}}\,dx'\to 0
$$
as $\lambda\to \lambda_1$.
Finally, by the decay law of $w(y),v(y)$ and $z(x)$, we conclude that there exists $\delta_0>0$, such that for all $\lambda\in[\lambda_1-\delta_0,\lambda_1]$, the following inequality holds
\begin{equation}\label{}
\begin{split}
&C_\lambda [\|w(y)\|_{L^{\frac{2N}{N-\alpha}}(\Sigma_\lambda^v)}\|v(y)\|^{\frac{4+\alpha-N}{N-2}}_{L^{\frac{2N}{N-2}}(\Sigma_\lambda^v)}+\|v(\bar x,0)\|_{L^{\frac{2(N-1)}{N-2}}(\partial \Sigma_\lambda^v)}^{\frac \alpha{N-2}}\|v(y)\|_{L^{\frac{2N}{N-2}}(\Sigma_\lambda^v)}^{\frac{2+\alpha}{N-2}}
\\&+\|z(x')\|_{L^{\frac{2(N-1)}{N-\alpha}}(\partial\Sigma_\lambda^v)}\|v(x',0)\|_{L^{\frac{2(N-1)}{N-2}}(\partial\Sigma_\lambda^v)}^{\frac{2+\alpha-N}{N-2}}
]<\frac 12.
\end{split}
\end{equation}
Hence, we infer from the above inequality and Lemma \ref{t 2.2} that $v(y)\leq v(y^\lambda)$ for all $\lambda\in[\lambda_1-\delta_0,\lambda_1]$. Finally, it follows from Lemma \ref{t 2.1} that $w(y)\leq w(y^\lambda)$ and $z(x')\leq z(x'^\lambda)$ for all $\lambda\in[\lambda_1-\delta_0,\lambda_1]$. This completes the proof of the claim.

\end{proof}

\begin{proposition}\label{t 3.3}
Suppose that $u,m,n$ is the positive solution of equation \eqref{2.1} and the nonlinear terms satisfy the assumptions in Theorem \ref{t 1.1}. For any $p\in \partial\mathbb R^{N}_+$, let $v,w,z$ be the Kelvin transformation of $u,m,n$ at $p$, then $v,w,z$ is symmetric with respect to any plane which is parallel to $x_N$ direction and passing through some point $q\in \partial \mathbb R^N_+$. Further more, if at least one of the functions $h,k,H,K$ is not a constant in $(0,\sup_{x\in \mathbb R_+^N}u(x))$ or $(0,\sup_{x\in \partial\mathbb R_+^N}u(x))$, then $u,m,n$ is symmetric with respect to any plane which is parallel to $x_N$ direction and passing through $p$, i.e., $p=q$.
\end{proposition}
\begin{proof}
To prove that $v,w,z$ is symmetric with respect to some point $q\in \partial\mathbb R_+^N$, we use the method of moving plane as above.
We prove the symmetry in every direction that is vertical to the
$x_N$ direction. Without loss of generality, we assume that $p=0$ and choose the $x_1$
direction and prove that $v,w,z$ is symmetric in the $x_1$ direction. We
can carry out the procedure as the above and assume that the plane stops at $\lambda_1$. If $\lambda_1>0$, then it
follows from Lemma \ref{t 3.2} that $v,w,z$ is symmetric with respect to $T_{\lambda_1}$. Otherwise, if
$\lambda_1\leq 0$, then we conclude by continuity that $v(x)\leq
v^{0}(x),w(x)\leq
w^{0}(x)$ and $z(x')\leq
z^{0}(x')$ for all $x\in \Sigma_0$ and $x'\in \partial \Sigma_0$. We can also perform the moving
plane procedure from the left and find a corresponding $\lambda_1'$.
If $\lambda_1'\geq 0$, then we get $v^{0}(x)\leq v(x),w^{0}(x)\leq w(x)$ and $z^{0}(x')\leq z(x')$ for $x\in
\Sigma_0$ and $x'\in \partial\Sigma_0$. This fact and the above inequality imply that $v(x),w(x),z(x')$ are
symmetric with respect to $T_0$. Otherwise, if $\lambda_1'<0$, an analogue to
Lemma \ref{t 3.2} shows that $v,w,z$ are symmetric with respect to
$T_{\lambda_1'}$.

Therefore, if $\lambda_1=\lambda_1'= 0$ for all
directions that is vertical to the $x_N$ direction and for all $p\in
\partial\mathbb R_+^N$, then $v,w,z$ and hence $u,m,n$ are symmetric with respect to any plane
which is parallel to the $x_N$ direction and passing through $p$.
Since $p$ is arbitrary, we conclude that $u,m,n$ depends only on $x_N$.

On the other hand, if $\lambda_1>0$ or $\lambda_1'<0$ in some
direction for some $p\in  \partial\mathbb R_+^N$, then we have $v=v^{\lambda_1},w=w^{\lambda_1},z=z^{\lambda_1}$ or
$v=v^{\lambda_1'},w=w^{\lambda_1'},z=z^{\lambda_1'}$. This implies that $v,w,z$ are regular at the origin,
and hence $u,m,n$ are regular at infinity. Since $v=v^{\lambda_1},w=w^{\lambda_1},z=z^{\lambda_1}$ or
$v=v^{\lambda_1'},w=w^{\lambda_1'},z=z^{\lambda_1'}$, then we infer from equation \eqref{2.5} that $h,k,H,K$ are constants.
\end{proof}

After the above preparations, we can prove Theorem \ref{1.1} now.
\begin{proof}[Proof of Theorem \ref{t 1.1}:]
Since at least one of the functions $h,k,H,K$ is not a constant in $(0,\sup_{x\in \mathbb R_+^N}u(x))$ or $(0,\sup_{x\in \partial\mathbb R_+^N}u(x))$, then
Proposition \ref{t 3.3} implies that the Kelvin's transform $v,w,z$ of
$u,m,n$ at $p$ is symmetric with respect to any plane which is parallel
to $x_N$ direction and passing through $p$ for any $p\in  \partial\mathbb R_+^N$.
Since $p$ is arbitrary, we conclude that $u,m$ depend only on $x_N$ and $n$ is a constant.
Then we deduce from equation \eqref{1.1} that
$$
\left\{
  \begin{array}{ll}
  \displaystyle
-\frac{d^2u(y_N)}{dy_N^2}=\intpr \frac{
F(u(0))}{|(x',0)-y|^{N-\alpha}}dx'g(u(y_N))    &{\rm in}\quad \mathbb R^+,    \\
\\ \displaystyle
 -\frac{\partial u}{\partial x_N}(0)=\intr \frac{G(u(y_N))}{|(x',0)-y|^{N-\alpha}}\,dy f(u(0)).
\end{array}
\right.
$$
The first equation implies that $u$ is concave in $x_N$ direction. It also follows from
the second equation that $\frac{d u}{d x_N}(0)=-\intr \frac{G(u(y_N))}{|(x',0)-y|^{N-\alpha}}\,dy f(u(0))\leq 0$. So
we deduce that $u$ is strictly decreasing in $x_N$ unless $u\equiv
c$ with $F(c)=G(c)=0$. If $u\equiv c$, then we proved our result. On
the other hand, if $u$ is strictly decreasing and concave in $x_N$,
then $u(x_N)<0$ for $x_N$ large enough, which contradicts that $u$
is a nonnegative solution for problem \eqref{1.1}.

\end{proof}

{\bf Acknowledgement}:The author would like to thank the anonymous referee for his/her careful reading and useful suggestions of this paper, which indeed improve this paper. This work is supported by Excellent Young Scientists Foundation of Guangdong Province, No. YQ2014154.

\end{document}